\documentclass{siamart250211}
\usepackage{amsmath,amssymb,amscd,amsfonts}
\usepackage{color}
\usepackage{graphicx}
\usepackage{mathtools}     
\usepackage{xfrac}
\usepackage{algpseudocode}
\usepackage{tikz}

\newtheorem{assumption}{Assumption}[section]
\newcommand{\dist}{\operatorname{dist}}

\newcommand{\pyr}{\operatorname{pyr}}
\newcommand{\aff}{\operatorname{aff}}
\newcommand{\conv}{\operatorname{conv}}

\newcommand{\relint}{\operatorname{relint}}

\newcommand{\xb}{\mathbf{x}}
\newcommand{\yb}{\mathbf{y}}
\newcommand{\zb}{\mathbf{z}}
\newcommand{\vb}{\mathbf{v}}

\newcommand{\wb}{\mathbf{w}}
\newcommand{\ab}{\mathbf{a}}
\newcommand{\bb}{\mathbf{b}}
\newcommand{\Ob}{{\mathbf{0}}}
\newcommand{\sigmab}{{\boldsymbol{\sigma}}}
\newcommand{\xib}{{\boldsymbol{\xi}}}
\newcommand{\etab}{{\boldsymbol{\eta}}}
\newcommand{\R}{\mathbb{R}}
\newcommand{\C}{\mathbb{C}}

\newcommand{\I}{\mathcal{I}}

\newcommand{\abs}[1]{\left| #1 \right|}

\newcommand{\leaves}[1]{\mathcal{L}(#1)}
\newcommand{\pyralat}[1]{\mathcal{P}(#1)}
\newcommand{\faces}[2]{\mathcal{F}(#1,#2)}
\newcommand{\singular}[1]{\mathcal{S}^{({#1})}}
\newcommand{\bases}[1]{\mathcal{B}^{({#1})}}
\newcommand{\paths}[1]{\mathrm{paths}(#1)}

\title{Quadrature for Singular Integrals over convex Polytopes}
\author{Johannes Tausch\thanks{
    Southern Methodist University,Dallas,TX, USA
    (\email{tausch@smu.edu})} }

\headers{Quadrature for Singular Integrals over convex Polytopes}{J.Tausch}

\begin{document}

\maketitle

\begin{abstract}
  A new algorithm for the efficient numerical approximation of weakly
  singular integrals over convex polytopes is introduced. Such
  integrals appear in the Galerkin discretizations of integral
  equations and nonlocal partial differential equations. The
  polytope is decomposed into a number of convex hulls of a singular and regular
  face. This expresses the singularity in a single variable which is
  effectively handled by Gauss-Jacobi quadrature. The decomposition
  algorithm is applicable to general finite polytopes. The Cartesian product
  of two simplices and two cubes will be discussed as special cases
  and numerical examples will be presented to illustrate the
  convergence of the resulting quadrature scheme.
\end{abstract}



\section{Introduction}
This paper is concerned with the numerical evaluation of integrals of the form
\begin{equation}\label{def:kernelintgr}
\int_{P_x} \int_{P_y} k(\xb,\yb) d\yb d\xb.
\end{equation}
Here $P_x,P_y\subset \R^d$ are
convex polytopes (triangles, quadrilaterals, tetrahedra, cubes, prisms, etc), 
and the kernel $k: \R^d\times \R^d\to \R$
is smooth when $\xb \not= \yb$ but singular when $\xb = \yb$.

Integrals of the type \eqref{def:kernelintgr} arise in boundary
integral reformulations of elliptic PDEs \cite{hsiao-wendland08,
mohyaddin-tausch23, sauter-schwab11, seibel23,tausch22}, parabolic and
hyperbolic PDE \cite{mason-tausch19,poltz-schanz19} and fractional
PDEs \cite{bonito-lei-pasciak19, delia-etal20, feist:bebebdorf23}.
Besides differential equations, integrals of the type
\eqref{def:kernelintgr} also arise in the evaluation of the Riesz
minimal energy \cite{of:wend:zorii10,harbr:wendl:zorii16} and in
certain applications in quantum mechanics
\cite{griebel:hamaekers07}.

We assume that $P_x$ and $P_y$ are conforming in the sense that either
$P_x = P_y$ or that the intersection $P_x \cap P_y$ is a common
face. Because of the singularity of the integrand it is challenging to
find suitable quadrature methods, and is a topic of current research.

The original motivation to devise efficient quadrature schemes for
\eqref{def:kernelintgr} came from Galerkin boundary element methods,
where $P_x$ and $P_y$ are panels in three dimensional space.  For
the case of triangles and parallelograms Sauter and co-workers
constructed singularity removing transformations that map the four
dimensional unit cube to the Cartesian product $P_x \times P_y$, see
\cite{erichsen-sauter98,sauter-schwab11}. Later, this methodology
was generalized to tetrahedra and parallelotopes in arbitrary dimensions
in a sequence of papers by Chernov, Petersdorff and Schwab, see
\cite{chern-peterd-schwab11,chernov-schwab12,chernov-reinarz13,chern-peterd-schwab15}.

The construction in the aforementioned papers is specific to the type
of polytope, and is therefore limited to pairs of simplices and pairs
of parallelotopes. The primary goal of the present paper is develop a
framework to handle general convex and bounded polytopes. The
fundamental principle here is a simple algorithm that decomposes any
finite polytope into a collection of generalized pyramids. This
algorithm will then be applied to the Cartesian product
$P_x \times P_y$ which will result in a decomposition of convex hulls
of certain faces of $P_x$ and $P_y$. The singularities in the
integrals over the convex hulls can be handled with Gauss-Jacobi
rules.

The reminder of this paper is organized as
follows. Section~\ref{sec:polytope} provides some background material
on convex polytopes. Section~\ref{sec:int:convhull} discusses
integration of singular functions over the convex hull of two faces.
The following two Sections~\ref{sec:pyraalg} and \ref{sec:kernel}
describe the pyramidal decomposition algorithm for general polytopes
and its application to the case of Cartesian products of two
polytopes.  Section~\ref{sec:examples} presents an analysis for
simplices and cubes and compares the present method with the previously
published work for these special cases. Finally,
Section~\ref{sec:numresult} presents a few numerical illustrations of
the method.

\section{Convex Polytopes}\label{sec:polytope}
Convex polytopes are a well studied topic in discrete geometry.  The
standard references are \cite{grunbaum03} and \cite{ziegler95} and
comprehensive overviews of the computational aspects of polytopes are
given in~\cite{goodman-orourke04} and~\cite{kaibel-pfetsch02}. In this
section we collect some basic concepts that will be used later on.
The proofs for some of the less obvious statements will be provided.

\subsection{Basic Concepts}
There are two ways to
specify bounded convex polytopes. The first is to consider a polytope as the
intersection of half spaces bounded by affine hyperplanes. This leads
to a system of linear inequalities
\begin{equation*}
P = \left\{ \xb \in \R^d \colon A \xb \leq \mathbf{b} \right\} 
\end{equation*}
where the each row represents a half space. A point in $P$ that
achieves equality in $d$ independent rows is a vertex of $P$. A
$j$-face is the subset of $P$ that achieves equality in $d-j$
independent rows of the hyperplane representation.

In the second description one starts with the vertices and specifies
$P$ as their convex hull. That is
\begin{equation*}
  P = [\vb_0, \dots, \vb_n] :=
   \left\{ \sum_{i=0}^n \lambda_i \vb_i \colon \lambda_i \geq
   0, \sum_{i=0}^n \lambda_i = 1 \right\}.
\end{equation*}
Any face of $P$ is the convex hull of a subset
of the vertices of $P$. Faces are distinguished based on their 
dimension. The 0-faces are the vertices, the 1-faces
are the edges and the faces with dimension one less than the polytope
are facets.

The affine hull of some face $F = [\vb_{i_0}, \dots, \vb_{i_k}]$ of
$P$ is the
hyperplane of smallest dimension that contains $F$
\begin{equation*}
\aff(F) = \left\{ \sum_{j=0}^k \lambda_j \vb_{i_j} :\,
  \sum_{j=0}^k \lambda_j = 1 \right\}.
\end{equation*}
For $\xb_0 \in \aff(F)$ the set
\begin{equation*}
H_F = \aff(F) - \xb_0
\end{equation*}
is a linear space, which is independent of the choice of $\xb_0$. If
$F$ is a $j$-face then the dimension of $H_F$ is $j$. 

The face lattice is the set of all faces of a polytope which is
partially ordered by inclusion. It is represented by a
graph whose nodes correspond to faces and
edges represent the inclusion of a $(j-1)$-face in a $j$-face. Two
polytopes are combinatorially equivalent if they have isomorphic face
lattices.

\subsection{Iterated Pyramids}\label{sec:pyra} Pyramids are
an important special type of polytopes, where all vertices
except for one are in the same affine hyperplane. The vertex that is
not in the hyperplane is called the apex. The convex hull of the other
vertices is the base and is one of the facets of the
pyramid. The notation for a pyramid with apex $\vb_0$ and base
$B_0$ is 
\begin{equation*}
  P = \pyr(\vb_0, B_0) := \left\{ (1-\lambda) \vb_0 + \lambda \xb_0 :\;
    \xb_0\in B_0,\; \lambda \in [0,1] \right\}.
\end{equation*}
If the base is also a pyramid,
i.e., $B_0=\pyr(\vb_1, B_1)$, then we use the notation
$\pyr(\vb_1,\vb_2,B_2)$ to indicate such an iterated pyramid. By
recursion, one obtains a sequence of bases $B_j$ and apices $\vb_j$
and the $k$-fold pyramid
\begin{equation*}
  P = \pyr(\vb_0,\dots, \vb_k, B_{k}).
\end{equation*}
as long as $\vb_j \not\in\aff(B_{j})$, $j\in \{0,\dots,k\}$.
The 
convex hull of the apices is a $k$-dimensional simplex, denoted by
$A_k :=  [\vb_0,\dots,\vb_k]$.

If the final base of an
iterated pyramid is a vertex,
then $\pyr(\vb_0,\dots, \vb_k, \{\vb_{k+1}\})$ is the
$k+1$-dimensional simplex with
vertices $\vb_0,\dots,\vb_{k+1}$.

\subsection{Convex Hull of two Polytopes}
We now turn to the convex hull of two polytopes 
$A$ and $B$ in $\R^d$. It is defined as
\begin{equation}\label{def:convAB}
  \conv(A,B) = \Big\{ (1-\lambda) \xb_A + \lambda \xb_B :
    \lambda \in [0,1],\; \xb_A \in A,\; \xb_B \in B \Big\} 
\end{equation}
While this concept can be defined for any two polytopes, we are only interested in
situations where for $\xb\in \conv(A,B)$ the points $\xb_A$ and $\xb_B$
and the parameter $\lambda$ are uniquely defined. This turns out to be
true if the following condition is met.
\begin{assumption}\label{asu:convAB}
\begin{equation*}
  \aff(A) \cap \aff(B) = \emptyset
  \quad\mbox{and}\quad
  H_A \cap H_B = \{\mathbf{0} \}
\end{equation*}
\end{assumption}

\begin{lemma}\label{lem:convAB}
  If Assumption \ref{asu:convAB} holds, then $\xb_A \in A$ and
  $\xb_B \in B$
  and $\lambda \in [0,1]$ are unique. 
\end{lemma}
\begin{proof}
Suppose $\xb \in \conv(A,B)$ such that 
\begin{equation}\label{x:equal:y}
\xb = (1-\lambda) \xb_A + \lambda \xb_B = (1-\mu) \yb_A + \mu \yb_B
\end{equation}
with $\xb_A$, $\yb_A$ in $A$, $\xb_B$, $\yb_B$ in $B$, and $\lambda$,
$\mu$ in $[0,1]$. Assume first that $\lambda \not= \mu$. Rearranging
and dividing by $\lambda-\mu$ reveals that the point
\begin{equation*}
  \frac{1-\mu}{\lambda -\mu} \yb_A - \frac{1-\lambda}{\lambda -\mu} \xb_A
  = \frac{\lambda} {\lambda -\mu}\xb_B - \frac{\mu}{\lambda -\mu} \yb_B
\end{equation*}
is in $\aff(A)$ and in $\aff(B)$ since the coefficients on both
sides add up to unity. This contradicts the first part of Assumption \ref{asu:convAB}
and thus $\lambda = \mu$. In the latter case we obtain from
\eqref{x:equal:y} that
\begin{equation*}
(1-\lambda) (\yb_A - \xb_A) = \lambda (\xb_B - \yb_B).
\end{equation*}
Since the vectors $\yb_A - \xb_A$ and  $\xb_B - \yb_B$ are in $H_A$
and $H_B$, respectively, it follows from the second part of Assumption \ref{asu:convAB}
that $\xb_A = \yb_A$ and $\xb_B = \yb_B$ and the assertion is shown.
\end{proof}
We conclude the section with the following observation.

\begin{lemma}\label{lem:pyr:is:conv}
  If $A_k$ is the convex hull of the apices $\{\vb_0,\dots, \vb_k\}$
  of an iterated pyramid, then
  \begin{equation}\label{eq:pyr:is:conv}
    \pyr(\vb_0,\dots, \vb_k, B_k) = \conv(A_k,B_k).
  \end{equation}
  Further, $A_k$ and $B_k$ satisfy Assumption \ref{asu:convAB}.
\end{lemma}

\begin{proof}
The statement is trivial for $k=0$. We use induction to show that it
holds for any $k$.
Since $B_{k-1} = \pyr(\vb_k,B_k)$ it follows from the definitions and
the induction assumption that
\begin{equation*}
   \pyr(\vb_0,\dots, \vb_k, B_k)
    = \pyr(\vb_0,\dots, \vb_{k-1}, B_{k-1})
    = \conv( A_{k-1}, B_{k-1} )
  \end{equation*}
To demonstrate that the latter set is equal to $\conv( A_k, B_k )$ we need to
show that for all $\ab_{k-1} \in A_{k-1}$, $\bb_{k} \in B_{k}$
\begin{equation*}
 (1-\lambda_1) \ab_{k-1} + \lambda_1
 \Big( (1-\lambda_2)\vb_k + \lambda_2 \bb_{k} \Big)
 = (1-\mu_1) \Big( (1-\mu_2) \ab_{k-1} + \mu_2 \vb_k \Big) +
  \mu_1  \bb_{k}
\end{equation*}
holds for some $(\lambda_1, \lambda_2)$ and $(\mu_1, \mu_2) \in
[0,1]^2$. This is equivalent to the system
\begin{equation*}
\begin{aligned}
1 - \lambda_1 &= (1-\mu_1)(1-\mu_2)\\
 \lambda_1 (1 - \lambda_2) &= (1-\mu_1)\mu_2\\
\lambda_1\lambda_2 &= \mu_1.
\end{aligned}
\end{equation*}
Indeed, for $(\lambda_1, \lambda_2) \in [0,1]^2$ given, the solution
of the system is
\begin{equation*}
  \mu_1 = \lambda_1 \lambda_2 \quad\text{and}\quad
  \mu_2 = \frac{\lambda_1-\lambda_1\lambda_2}{1-\lambda_1 \lambda_2} 
\end{equation*}
which is again in $[0,1]^2$. Vice versa, for $(\mu_1, \mu_2) \in [0,1]^2$
\begin{equation*}
  \lambda_1 = \mu_1 + \mu_2 - \mu_1\mu_2 \quad\text{and}\quad
  \lambda_2 = \frac{\mu_1}{\mu_1 + \mu_2 -\mu_1 \mu_2} 
\end{equation*}
will do the job. This proves the first assertion \eqref{eq:pyr:is:conv}.

To show the second assertion, assume $\xb \in \aff(A_k) \cap
\aff(B_k)$. Denote the vertices of $B_k$ by $\wb_0,\dots,\wb_n$, then
\begin{equation}\label{def:x:proof}
\xb = \sum_{i=0}^k \lambda_i \vb_i = \sum_{i=0}^n \mu_i \wb_i, 
\end{equation}
where $\sum_i \lambda_i = \sum_i \mu_i = 1$. Assume first that
$\lambda_k \not= 1$. Rearranging and
dividing by $1 - \lambda_k$ leads to
\begin{equation}\label{sum:vw}
\sum_{i=0}^{k-1} \lambda_i^* \vb_i = \sum_{i=0}^n \mu_i^* \wb_i - \lambda_k^* \vb_k. 
\end{equation}
where $\lambda_i^* = \lambda_i/(1-\lambda_k)$ and $\mu_i^* = \mu_i/(1-\lambda_k)$.
The left hand side in \eqref{sum:vw} is an affine combination of
the vertices of $A_{k-1}$ and the right hand hand side is an affine
combination of vertices of $B_{k-1}$. This contradicts the
induction assumption and thus we are left to check $\lambda_k = 1$. In this case
it follows from \eqref{def:x:proof} that $\xb = \vb_k$ and $\xb \in
\aff(B_k)$ contradicting the assumption that the apex of the pyramid
$B_{k-1}$ is not in the affine plane of its base.

It remains to show that $H_{A_k} \cap H_{B_k} = \{ 0 \}$. To that end,
consider $\xb \in H_{A_k} \cap H_{B_k}$. Since $\xb \in H_{A_k}$ there are coefficients
$\lambda_i\in \R$ such that
\begin{equation}\label{intersection:H}
\xb = \sum_{i=1}^{k-1} \lambda_i (\vb_i-\vb_0) + \lambda_k(\vb_k -\vb_0).
\end{equation}
If $\lambda_k = 0$ then this implies that $\xb \in  H_{A_{k-1}} \cap
H_{B_k}$. This is a subspace of $\xb \in  H_{A_{k-1}} \cap
H_{B_{k-1}}$, which by induction assumption is trivial. Hence $\xb =
  0$. If $\lambda_k \not= 0$ we can rearrange \eqref{intersection:H} as
\begin{equation*}
  \vb_k - \frac{1}{\lambda_k} \xb =
  \vb_0 - \sum_{i=1}^{k-1} \frac{\lambda_i}{\lambda_k} (\vb_i-\vb_0).
\end{equation*}
Since $-\xb/\lambda_k \in H_{B_k} = \aff(B_k) - \vb_k$ it follows that
the left hand side is in $\aff(B_{k-1})$. On the other hand the right hand side
is in $\aff(A_{k-1})$ which contradicts the induction assumption. Thus the
zero vector is the only vector in $H_{A_k} \cap H_{B_k}$.
\end{proof}

\subsection{Distance Formula}
The distance of a point $\xb \in \R^d$ from $\aff(A)$ is given by
\begin{equation*}
  \dist(\xb,A) := \min\limits_{\ab \in \aff(A)} \abs{\xb-\ab}
  = \abs{ \Pi_A (\xb-\ab_0) }.
\end{equation*}
Here, $\abs{\cdot}$ is the Euclidean norm and $\Pi_A$ is the orthogonal projector into the subspace $H_A^\perp$,
and $\ab_0$ is an arbitrary point in $\aff(A)$.
If $\xb \in \conv(A,B)$ then a simple calculation with $\ab_0 = \xb_A$
shows that
\begin{equation}\label{dist:A}
  \begin{aligned}
  \dist(\xb,A) &=
  \abs{ \Pi_A\Big( (1-\lambda) \xb_A + \lambda \xb_B - \xb_A\Big) } = 
  \lambda \abs{ \Pi_A \Big( \xb_B - \xb_A\Big) }\\
    &= \lambda \dist(\xb_B,A).
    \end{aligned}
\end{equation}

\section{Integrals over Faces and Convex Hulls}\label{sec:int:convhull}
Integrals over a $j$-face can be carried out using an orthogonal basis
for the parameterization.
If the columns of the matrix $Q \in \R^{d\times j}$ are orthogonal
basis vectors of $H_F$ and $\vb \in \aff(F)$ then
\begin{equation}\label{def:intFace}
  \int_F f(\xb) d\xb  = \int_{F_Q} f( \vb + Q \xib))\, d\xib ,  
\end{equation}
where the parameter space $F_Q \subset \R^j$ is a polytope with
nonempty interior. If $F$ is
parameterized by a basis $T$ of $\aff(F)$ then  
\begin{equation*}
  \int_F f(\xb)\, d\xb  = \int_{F_T} f( \vb + T \xib))\, d\xib\, \det(R).  
\end{equation*}
Here $T=QR$ is the qr-factorization and $R \in \R^{j \times j}$ is upper triangular. 

We now turn to integrals over the convex hull of two polytopes $A$ and
$B$ that satisfy Assumption \ref{asu:convAB}.  The dimension of
$\conv(A,B)$ is $n=s+r+1$ where $s=\dim(A)$ and $r=\dim(B)$.
Further, if $\xb_A = \vb + T_A \xib$ and  $\xb_B = \wb + T_B \etab$ are
parameterizations of $A$ and $B$, 
then the Jacobian of $\xb = (1-\lambda) \xb_A + \lambda \xb_B$ is
\begin{equation*}
J = \Big[
(1-\lambda) T_A, \lambda T_B, \wb - \vb + T_B \etab - T_A \xib
\Big] \in \R^{d\times n}.
\end{equation*}
This matrix can be factored as
\begin{equation*}
  J = Q R
  \begin{bmatrix} (1-\lambda) I_s & & \\ & \lambda I_r \\ & & 1 \end{bmatrix},
\end{equation*}
where $Q \in \R^{d\times n}$ is orthogonal $R\in \R^{n\times n}$ is
upper triangular. 
Here only the last column of $R$ depends on $\xib$ and $\etab$, but
not its last entry. This follows from the fact that the last column of
$Q$ is in the orthogonal complement of the range of $[T_A,T_B]$. Thus
the determinant of $R$ is a constant. Reverting the parameterizations
of $A$ and $B$ one can see that
\begin{equation*}
  \int\limits_{\conv(A,B)}\!\!\! f(\xb)\, d\xb = \delta_{AB}
  \int_0^1 \int_{A}\int_{B}
  f\big( (1-\lambda) \xb_A + \lambda \xb_B \big)
    d\xb_B d\xb_A (1-\lambda)^{s} \lambda^{r} d\lambda\, 
\end{equation*}
where
\begin{equation}\label{def:deltaAB}
 \delta_{AB} = \frac{ \det(R)}{\det(R_A)\det(R_A)}
\end{equation}
and $R_A$ and $R_B$ are the triangular factors in the
qr-factorizations of $T_A$ and $T_B$. It is easy to verify that
$\delta_{AB}$ is independent of the choice of bases.  The particular
form of this integral can be exploited when the integrand is singular
on the face $A$ but smooth otherwise.  Specifically, we assume that
\begin{assumption}\label{asu:fcn}
There is $\alpha\in \R$ with $\alpha < r+1$ and a smooth function $g$ such that
\begin{equation*}  
f(\xb) = g(\xb) \dist(\xb, S)^{-\alpha}.
\end{equation*}
\end{assumption}
From the distance formula \eqref{dist:A} it follows that 
\begin{multline}\label{intgr:conv:hull}
  \int\limits_{\conv(A,B)}\!\!\! f(\xb)\, d\xb = \\ 
  \delta_{AB} \int_0^1 \int_{A}\int_{B}
  g\big( (1-\lambda) \xb_A + \lambda \xb_B \big) \dist^{-\alpha}(\xb_B,A)
    \,d\xb_B d\xb_A (1-\lambda)^{s} \lambda^{r-\alpha} d\lambda .
\end{multline}
Since $\dist(x_B,A)>0$ the $AB$-integral is smooth and the singularity
appears only in the $\lambda$ variable. To approximate \eqref{intgr:conv:hull}
by numerical integration, one can use quadrature rules for smooth
functions over the faces $A$ and $B$. The singularity in $\lambda$ can be treated
effectively by using the Gauss-Jacobi quadrature rule for the weight
function $w(\lambda) = (1-\lambda)^{s} \lambda^{r-\alpha}$.

\section{Pyramidal Decomposition Algorithm}\label{sec:pyraalg}
The strategy to handle singular integrals over arbitrary polytopes is
to construct a decomposition of the polytope into convex hulls and
apply integration formula \eqref{intgr:conv:hull}.  To that end,
let $P = [\vb_0,\dots,\vb_{n}]$ be a convex polytope where the
integrand is singular on the vertices $\{\vb_0,\dots,\vb_{k}\}$, where
$k<n$. Further, the integrand satisfies Assumption~\ref{asu:fcn} with
the the affine plane $S=\aff(\vb_0,\dots,\vb_k)$.  For the following
discussion, we make the following assumption about the faces of the
polytope
\begin{assumption}\label{asu:PS}
All faces $F$ of $P$ either have a singular vertex or satisfy
\begin{equation*}  
\dist(F,S) > 0.  
\end{equation*}
\end{assumption}
The assumption states that if $S$ intersects with a face then the
intersection contains at least one vertex of $F$.
If the convex hull of all singular vertices is a
face of $P$, then it is easy to see that Assumption~\ref{asu:PS} is
satisfied. However, the assumption includes more general situations. An
example of a polytope that is singular in its interior and still
satisfies the assumption is shown in Figure~\ref{fig:octagon}.

The basic building block of the pyramidal decomposition algorithm is
the simple observation that any polytope can be decomposed into
pyramids by selecting any vertex as the apex. To describe this fact
more formally, we denote by $\faces{\vb}{P}$ the set of facets of the
polytope $P$ that do not contain the vertex $\vb$. Then the following
lemma holds.
\begin{lemma}\label{lem:decomp:P}
If $\vb$ is a vertex of a finite convex polytope $P$ then
\begin{equation}\label{decomp:P}
  P = \bigcup\limits_{F \in \faces{\vb}{P}} \!\! \pyr(\vb, F),
\end{equation}
which is a union with disjoint relative interiors.
\end{lemma}

\begin{proof}
For a point $\xb \in \relint(P)$ consider the line that contains both
$\xb$ and the vertex $\vb$. Since $P$ is finite, the intersection of
this line with $P$ is a line segment $[\vb,\wb]$, where $\wb$ is a
point on a facet $F$ of $P$. If this facet contained the vertex $\vb$
then by convexity of $F$ the entire line segment $[\vb,\wb]$ would be
in $F$ contradicting the assumption that $\xb$ is in the
interior. Hence $F \in \faces{\vb}{P}$ and $\xb \in
\pyr(\vb,F)$. Taking the closure shows that the left hand side in \eqref{decomp:P}
is contained in the right hand side.
Since the opposite inclusion is obvious the equality in
\eqref{decomp:P} follows.

To show the second statement note that for two distinct facets
$F,F' \in \faces{\vb}{P}$
$\pyr(\vb, F) \cap \pyr(\vb, F') = \pyr(\vb, F\cap F')$. This is a
lower dimensional face of both pyramids, hence their interiors are
disjoint.
\end{proof}

\subsection{Decomposition Algorithm}
The strategy is to select a singular vertex of $P$ and break up $P$
into pyramids as in \eqref{decomp:P}. If in this decomposition a base
has no further singular vertices, the decomposition is stopped.  If a
base in \eqref{decomp:P} contains another singular vertex, then the
base will be split using \eqref{decomp:P} with the next singular
vertex as the apex. The process is repeated until all bases have no
singular vertex. The resulting recursive algorithm is described
in~\ref{algo:pyradecomp}. We summarize some obvious facts about
Algorithm~\ref{algo:pyradecomp}.

\begin{algorithm}
\caption{Pyramidal Decomposition Algorithm}
\label{algo:pyradecomp}
\begin{algorithmic}
\Function {pyraDecomp}{$F$}
\If{$F$ has no singular vertices}
  \State \Return
  \Else
  \State select a singular vertex $\vb$
  \ForAll{$F' \in \faces{\vb}{F}$}
     \State \Call{pyraDecomp}{$F'$}
  \EndFor
\EndIf
\EndFunction
\end{algorithmic}
\end{algorithm}

\begin{enumerate}
\item The bases generated by the algorithm form a partially
  ordered set, hence their dependence can be expressed by a lattice, 
  which will be referred to as the pyramidal lattice
  $\pyralat{P}$. It is obtained from the face lattice by deleting the faces that
  are not in some $\faces{\vb}{F}$ generated during 
  Algorithm~\ref{algo:pyradecomp}.
\item In every step of the pyramidal decomposition, the dimension of
  the faces is reduced by one. We say that a face in the
  pyramidal lattice is in level $\ell$ if its dimension is
  $\dim{P}-\ell$. Moreover, the dimension of the intersection of a
  face with $S$ is reduced by at least one. Hence, the number of levels in
  $\pyralat{P}$ is at most $\dim(S)+2$.
\item If $B$ is a nonsingular base in the $\ell+1$-st level of
  $\pyralat{P}$, and $\vb_{\sigma_0}, \dots, \vb_{\sigma_\ell}$ are
  the apices on a path to $B$ then
  $\pyr(\vb_{\sigma_0}, \dots, \vb_{\sigma_\ell}, B)$ is one of the
  iterated pyramids generated by the algorithm. From
  Lemma~\ref{asu:convAB} it follows that
  $[\vb_{\sigma_0}, \dots, \vb_{\sigma_\ell}]$ is an
  $\ell$-dimensional simplex.  Since it is possible that there are
  multiple paths to $B$ we denote by $\paths{B}$ the set of all paths
  to $B$. Further, we denote by $\leaves{P}$ the set of all
  nonsingular bases in all levels of $\pyralat{P}$. Then
\begin{equation}\label{pyradecomp}
P = \bigcup_{B \in \leaves{P}} \bigcup_{ \vb_{\sigma_0},
  \dots, \vb_{\sigma_\ell}\atop \in \paths{B}} \pyr(\vb_{\sigma_0},
  \dots, \vb_{\sigma_\ell}, B)
\end{equation}    
is a decomposition with disjoint interiors. 
\item Because of Assumption~\ref{asu:PS} all bases in $\leaves{P}$ are
  nonsingular. Thus singular integrals over the iterated pyramids can
  be treated with the method of Section~\ref{sec:int:convhull}.
\item If a face encountered in Algorithm~\ref{algo:pyradecomp} has
  multiple singular vertices the choice of apex is not unique and
  hence the resulting decomposition of $P$ into iterated pyramids is
  not unique.
\item The number of sets in the right hand side of \eqref{pyradecomp}
  can be reduced if the union of simplices on a path to a base $B$
  form a convex polytope
\begin{equation*}
  A_B = \bigcup_{ \vb_{\sigma_0},
  \dots, \vb_{\sigma_\ell}\atop \in \paths{B}} [\vb_{\sigma_0},
  \dots, \vb_{\sigma_\ell}].
\end{equation*}
In this case it follows that
\begin{equation*}
\bigcup_{ \vb_{\sigma_0},
  \dots, \vb_{\sigma_\ell}\atop \in \paths{B}} \pyr(\vb_{\sigma_0},
  \dots, \vb_{\sigma_\ell}, B) = \conv(A_B,B).
\end{equation*}    
Since the affine plane of $A_B$ is the same as the affine plane of any
of the simplices in $\paths{B}$, Assumption~\ref{asu:convAB} holds for
$A_B$ and $B$ and hence the singularity in the integral over $\conv(A_B,B)$ can 
be treated as in \eqref{intgr:conv:hull}.
\end{enumerate}

\subsection{An Example}\label{sec:exampleOcta}
We illustrate the pyramidal decomposition algorithm with the double
pyramid shown in Figure~\ref{fig:octagon}. Here the vertices
$\vb_0,\vb_1,\vb_2,\vb_3$ are the singular vertices which span the
two-dimensional plane affine plane $S$. The vertices $\vb_4,\vb_5$ 
are on either side of this plane. All one- and two-dimensional faces
of $P$ contain a singular vertex, thus one can see that
Assumption~\ref{asu:PS} is indeed satisfied. Beginning Algorithm~\ref{algo:pyradecomp}
with  $\vb_0$ as the first singular vertex, results in the pyramidal
lattice shown in the figure. 

\begin{figure}
\begin{center}
\begin{tikzpicture}[line join=bevel,z=-3.3,scale=2.3]
\coordinate (3) at (0,0,-1);
\coordinate (0) at (-1,0,0);
\coordinate (1) at (0,0,1);
\coordinate (2) at (1,0,0);
\coordinate (4) at (0,1,0);
\coordinate (5) at (0,-1,0);
\draw[thick] (0)--(1);
\draw[thick] (1)--(2);
\draw[dashed] (2)--(3);
\draw[dashed] (3)--(0);
\draw[thick] (0)--(4);
\draw[thick] (1)--(4);
\draw[thick] (2)--(4);
\draw[dashed] (3)--(4);
\draw[thick] (0)--(5);
\draw[thick] (1)--(5);
\draw[thick] (2)--(5);
\draw[dashed] (3)--(5);
\node [left] at (0) {$0$};
\node [below left] at (1) {$1$};
\node [right] at (2) {$2$};
\node [above right] at (3) {$3$};
\node [left] at (4) {$4\;$};
\node [left] at (5) {$5$};
\end{tikzpicture}
\hspace{1cm}
\begin{tikzpicture}[xscale=1.4, yscale=1.4]
  \node (A) at (2.5,3){$[0 1 2 3 4 5]$};
  \node (B) at (1,2){$[1 2 4]$};
  \node (C) at (2,2){$[1 2 5]$};
  \node (D) at (3,2){$[2 3 4]$};
  \node (E) at (4,2){$[2 3 5]$};
  \node (F) at (1,1){$[2 4]$};
  \node (G) at (2,1){$[2 5]$};
  \node (H) at (3,1){$[3 4]$};
  \node (I) at (4,1){$[3 5]$};
  \node (J) at (2,0){$[4]$};
  \node (K) at (3,0){$[5]$};
  \draw[-,thick] (A)--(B);
  \draw[-,thick] (A)--(C);
  \draw[-,thick] (A)--(D);
  \draw[-,thick] (A)--(E);
  \draw[-,thick] (B)--(F);
  \draw[-,thick] (C)--(G);
  \draw[-,thick] (D)--(H);
  \draw[-,thick] (E)--(I);
  \draw[-,thick] (F)--(J);
  \draw[-,thick] (H)--(J);
  \draw[-,thick] (G)--(K);
  \draw[-,thick] (I)--(K);
\end{tikzpicture}
\end{center}
\caption{Double pyramid with singular vertices $\vb_0,\vb_1,\vb_2,\vb_3$ and regular
  vertices $\vb_4$ and $\vb_5$ (left) and the resulting
  pyramidal lattice. The first index in a face denotes
the apex for the next pyramid.} 
\label{fig:octagon}
\end{figure}
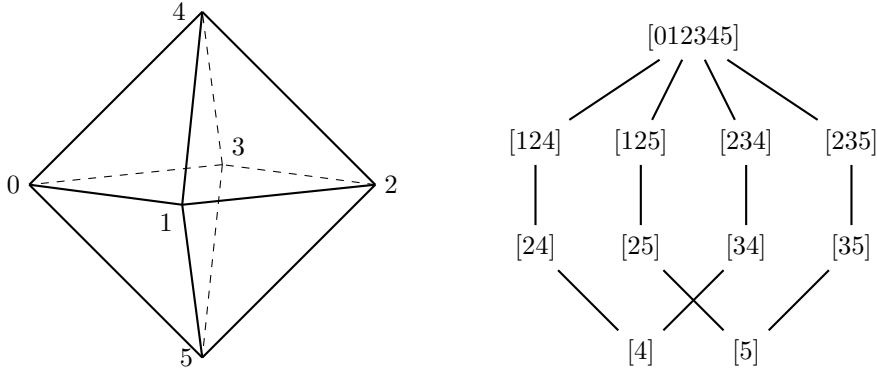

In this lattice there are two paths each to the nonsingular bases
$[\vb_4]$ and $[\vb_5]$. This implies that the double pyramid
has been decomposed into four iterated pyramids
\begin{multline}\label{decomp:octa:example}
  P =  \pyr(\vb_0,\vb_1,\vb_2,[\vb_4]) \;\cup\;
        \pyr(\vb_0,\vb_2,\vb_3,[\vb_4])\;\cup \\
       \pyr(\vb_0,\vb_1,\vb_2,[\vb_5]) \;\cup\;
        \pyr(\vb_0,\vb_2,\vb_3,[\vb_5]).
\end{multline}
All pyramids in this decomposition are simplices.
It is natural to ask whether \eqref{decomp:octa:example} is the
decomposition with the smallest possible number of pyramids with
nonsingular bases. In general, the choice of the apices in
Algorithm~\ref{algo:pyradecomp} can affect the
final number, but in this simple example 
that starting with another vertex will always result in
the same number of pyramids. 

However, the number of sets can be reduced if a decomposition into
convex hulls is sought. Such a decomposition can also be read off from the
pyramid lattice in Figure~\ref{fig:octagon}. The apices of the two
paths that end in $\vb_4$ form a triangulation of the singular face
$[\vb_0, \vb_1, \vb_2, \vb_3]$. Thus we can combine these two pyramids
as follows
\begin{equation*}
\begin{aligned}
\pyr(\vb_0,&\vb_1,\vb_2,[\vb_4]) \;\cup\; \pyr(\vb_0,\vb_2,\vb_3,[\vb_4])\\
&= \conv([\vb_0,\vb_1,\vb_2],[\vb_4]) \;\cup\; \pyr([\vb_0,\vb_2,\vb_3],[\vb_4])\\ 
&= \conv([\vb_0,\vb_1,\vb_2,\vb_3],[\vb_4])
\end{aligned}
\end{equation*}     
Doing the same for the path to base $[\vb_5]$ shows that
\begin{equation*}
\begin{aligned}
  P = \conv([\vb_0,\vb_1,\vb_2,\vb_3],[\vb_4]) \;\cup\;
  \conv([\vb_0,\vb_1,\vb_2,\vb_3],[\vb_5]).
\end{aligned}
\end{equation*}     


\subsection{Triangulations}\label{sec:triangulations}
Algorithm~\ref{algo:pyradecomp} can be modified by selecting any
vertex of a face instead of a singular vertex and continuing until the
final base consists of only one vertex. The result is a triangulation,
that is, the polytope is broken up into simplices. Triangulations of
polytopes have been studied extensively, see,
e.g.~\cite{goodman-orourke04, kaibel-pfetsch02}. It is important to
remark here that Algorithm~\ref{algo:pyradecomp} is not designed to
minimize the number of simplices. In fact, all simplices generated by
Algorithm~\ref{algo:pyradecomp} contain the initial vertex $\vb_0$,
whereas in an optimal triangulation it is not necessarily true that all
simplices have a common vertex.

Since this will be used in Section~\ref{sec:twoCubes} below, we
illustrate the triangulation of the $d$-cube by
Algorithm~\ref{algo:pyradecomp}. The $d$-dimensional cube is the $d$-fold
Cartesian product $C=E\times\dots\times E$, where $E = [0,1]$ the
one-dimensional cube. The vertices are
$\vb_{\sigmab} = (\sigma_1,\dots,\sigma_d)$, where
$\sigma_i \in \{0,1\}$, and the faces are
$F = F_1\times\dots\times F_d$, where $F_i \in \{ \{0\},\{1\}, E \}$.
The dimension is $\dim(F) = \# \{i: F_i = E \}$.  If
$v_\Ob = (0,\dots,0)$ is the initial apex, then
$\faces{C}{v_\Ob}$ is the set of faces $\bases{d-1}$ where
\begin{equation*}
\bases{j} = \{ F_1\times\dots\times F_d \,:\, F_i \in  \{ \{1\}, E \}
\mbox{ and } \dim(F) = j \},\quad j\in \{0,\dots,d\}.
\end{equation*}
For the face $F_1\times\dots\times F_d \in \bases{j}$ the apex
$\vb_\sigmab$ is chosen according to the rule
\begin{equation*}
  \sigma_i = \begin{cases}
    0 & \mbox{if } F_i = E, \\ 1  &\mbox{if } F_i = \{1\}.
    \end{cases}
\end{equation*}
Thus $\faces{F}{\vb_\sigmab}\subset \bases{j-1}$ and $\#\faces{F}{\vb_\sigmab} = j$.
The pyramidal lattice for
$d=3$ is shown in Figure~\ref{fig:onecube}.

\begin{figure}
\begin{center}
\begin{tikzpicture}[xscale=0.8, yscale=0.8]
  \node (A) at (4,6){$E\times E\times E$};
  \node (B) at (0,4){$\{1\}\times E\times E$};
  \node (C) at (4,4){$E\times \{1\}\times E$};
  \node (D) at (8,4){$E\times E\times \{1\}$};
  \node (E) at (0,2){$\{1\}\times \{1\}\times E$};
  \node (F) at (4,2){$\{1\}\times E\times \{1\}$};
  \node (G) at (8,2){$E\times \{1\}\times \{1\}$};
  \node (H) at (4,0){$\{1\}\times \{1\}\times \{1\}$};
  \draw[-,thick] (A)--(B);
  \draw[-,thick] (A)--(C);
  \draw[-,thick] (A)--(D);
  \draw[-,thick] (B)--(E);
  \draw[-,thick] (B)--(F);
  \draw[-,thick] (C)--(E);
  \draw[-,thick] (C)--(G);
  \draw[-,thick] (D)--(F);
  \draw[-,thick] (D)--(G);
  \draw[-,thick] (E)--(H);
  \draw[-,thick] (F)--(H);
  \draw[-,thick] (G)--(H);
\end{tikzpicture}
\end{center}
\caption{Pyramidal lattice for the cube.} 
\label{fig:onecube}
\end{figure}
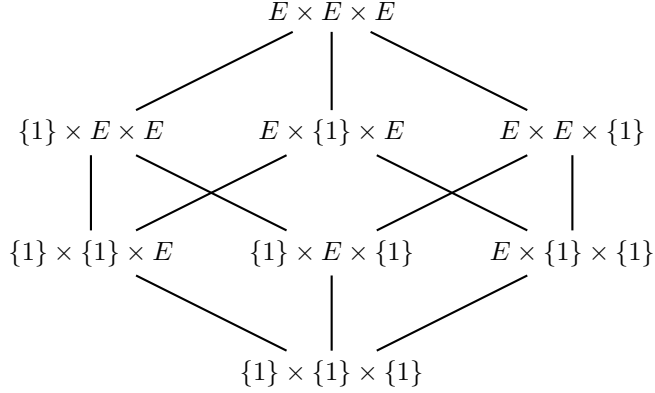

Since $\bases{0}$ consists only of the vertex $\vb_{(1\dots 1)}$, the simplices in the
triangulation of the cube are obtained by following all paths in the
pyramidal lattice from $C$ to $\vb_{(1\dots 1)}$. The resulting simplices
are of the form $[\vb_{\sigmab_0}, \dots \vb_{\sigmab_d}]$, where
\begin{equation*}
  \sigmab_j = (\sigma_{j 1},\dots \sigma_{j d}),\quad
  \sum \sigma_{ji} = j, \quad \mbox{and } \sigma_{j+1,i} \geq \sigma_{ji}\,. 
\end{equation*}
Since a $j$-dimensional face is decomposed into $j$ pyramids, the
number of paths and thus the number of simplices in the final decomposition
is  $d!$.

\section{Singular Kernel Integrals}\label{sec:kernel}
We now turn to the main motivation of this work which is to construct
a suitable decomposition 
for the integral \eqref{def:kernelintgr}. Here the integrand must satisfies
the following assumption.
\begin{assumption}\label{asu:kernel}
There is $\alpha \in \R$ and a smooth function $g$ such that
\begin{equation*}
k(\xb,\yb) = \abs{\xb - \yb}^{-\alpha} g(\xb,\yb).
\end{equation*}
\end{assumption}
The power $\alpha$ can be a real number, but should be such that the
integral \eqref{def:kernelintgr} is at most weakly singular.
The polytopes $P_x$ and $P_y$ are in $\R^d$ but can have
different dimensions (e.g., a tetrahedron and a triangle in
$\R^3$. However, they must be conforming in the following sense.
\begin{assumption}\label{asu:Pconforming}
$P_x$ and $P_y$ intersect in a common face.
\end{assumption}
This assumption includes the case that $P_x = P_y$.
The polytopes are the convex hull of their vertices
\begin{equation}\label{def:prdPoly}
P_x = [\vb_0,\dots,\vb_s, \vb_{s+1},\dots \vb_n] \; \mbox{and}\;
P_y = [\vb_0,\dots,\vb_s, \wb_{s+1},\dots \wb_m]
\end{equation}
where $\big\{\vb_k,\; k\in\{1,\dots,s\} \big\}$ are the vertices of the common
face. In the following we will also use the notation $\wb_k = \vb_k$ for $k
\in\{1,\dots,s\}$ to denote the common vertices.

The domain of integration in \eqref{def:kernelintgr} is the Cartesian product
\begin{equation*}
P = P_x \times P_y = \Big \{(\xb,\yb)\,:\, \xb \in P_x,\; \yb\in P_y \Big\},
\end{equation*}
which itself is a convex polytope. The faces of $P$ are Cartesian products of
the faces of $P_x$ and $P_y$ and the vertices are 
\begin{equation*}
  (\vb_k,\wb_\ell),
  \quad k\in\{1,\dots m\},\; \ell \in\{1,\dots n\}.
\end{equation*}
Because of Assumption~\ref{asu:kernel} the singular vertices are
\begin{equation}\label{singular:vtxs}
  (\vb_k,\vb_k), 
  \quad k\in\{1,\dots s\},
\end{equation}
and the singular plane is the subspace
\begin{equation*}
  S = \aff\{ (\vb_k,\vb_k) \}_{k\in\{1,\dots s\}} \subset \{ (\xb,\xb) : \xb \in \R^d \}.
\end{equation*}

The idea is to apply the pyramidal decomposition Algorithm~\ref{algo:pyradecomp}
to the polytope $P$. To that end we have to ensure that
the necessary assumptions in the previous section are satisfied for
the Cartesian product.

\begin{lemma}
If $P_x$ and $P_y$ intersect in a common face then $P$
satisfies Assumption~\ref{asu:PS}. 
\end{lemma}
\begin{proof}
Consider the face $F_x\times F_y$ of $P$ and let $G=F_x\cap
F_y$. If $G$ is empty then 
$0 < \dist(F_x,F_y) \leq \dist(F_x\times F_y, S)$. If 
$G$ is nonempty then it is
a common face of $F_x$ and $F_y$. Then
a vertex $\vb$ of $G$ is a vertex of $F_x$ and $F_y$. 
Thus $(\vb,\vb)$ is a singular vertex of $F_x\times F_y$. This implies
Assumption~\ref{asu:PS}. 
\end{proof}
The following is an application of Lemma~\ref{lem:decomp:P} to
Cartesian products of finite polytopes. 

\begin{lemma}\label{lem:decomp:PxPy}
If $\vb$ is a vertex of $P_x$ and $\wb$ is a vertex of $P_y$, then
\begin{equation*}
  \faces{(\vb,\wb)}{P_x\times P_y} =
  \left\{F_x\times P_y : F_x \in \faces{\vb}{P_x} \right\} \cup
   \left\{P_x\times F_y : F_y \in \faces{\wb}{P_y} \right\}
\end{equation*}
and
\begin{equation*}
  P_x \times P_y = \!\!
  \bigcup\limits_{F_x \in \faces{\vb}{P_x}} \!\! \pyr((\vb,\wb), F_x
  \times P_y) \;\;\cup\!\!
 \bigcup\limits_{F_y \in \faces{\wb}{P_y}} \!\! \pyr((\vb,\wb), P_x
  \times F_y).
\end{equation*}
\end{lemma}
\begin{proof}
The first assertion follows from the fact that the facets of $P_x
\times P_y$ are Cartesian products of the form $F_x \times P_y$ and
$P_x \times F_y$, where $F_x$ and $F_y$ are facets of $P_x$ and $P_y$.
Moreover, these products do not contain $(\vb,\wb)$ if and only if 
$F_x\in \faces{\vb}{P_x}$ and $F_y\in \faces{\vb}{P_y}$. 

For the second assertion note that $(\vb,\wb)$ is not in
$\aff(F_x \times P_y)$ or $\aff(P_x \times F_y)$, so the sets in this
decomposition are indeed pyramids. The statement then follows from
Lemma~\ref{lem:decomp:P}, setting $P = P_x \times P_y$.
\end{proof}

The decomposition Algorithm~\ref{algo:pyradecomp} can now be
applied to Cartesian products of polytopes. Each step involves a
product of faces of $P_x$ and $P_y$. If the two faces have a common
vertex, then the product is split into pyramids using the result of 
Lemma~\ref{lem:decomp:PxPy}. If the faces have no common vertex, the
recursion is terminated.

Figure~\ref{fig:TriangQuad} illustrates the algorithm for the product
of a triangle and a convex quadrilateral that have a common edge.  One
can see that the decomposition results in six pyramids.

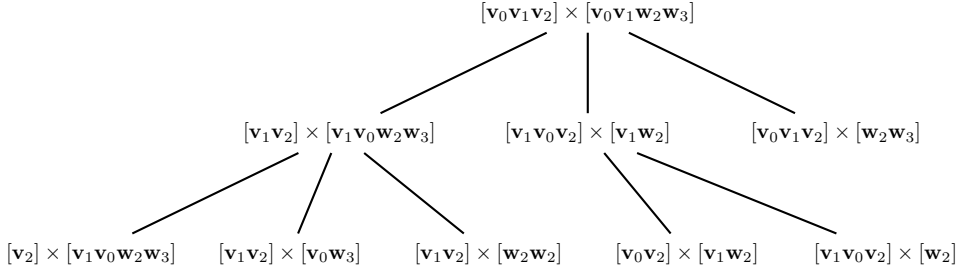
\begin{figure}
\begin{center}
\begin{tikzpicture}[xscale=1.65, yscale=1.6]
  \node (A) at (4,2){\scalebox{0.8}{$[\vb_0 \vb_1 \vb_2] \times [\vb_0 \vb_1 \wb_2 \wb_3]$}};
  \node (B) at (2,1){\scalebox{0.8}{$[\vb_1 \vb_2] \times [\vb_1 \vb_0 \wb_2 \wb_3]$}};
  \node (C) at (4,1){\scalebox{0.8}{$[\vb_1 \vb_0 \vb_2] \times [\vb_1 \wb_2]$}};
  \node (D) at (6,1){\scalebox{0.8}{$[\vb_0 \vb_1 \vb_2] \times [\wb_2 \wb_3]$}};
  \node (E) at (0,0){\scalebox{0.8}{$[\vb_2] \times [\vb_1 \vb_0 \wb_2 \wb_3]$}};
  \node (F) at (1.6,0){\scalebox{0.8}{$[\vb_1 \vb_2] \times [\vb_0 \wb_3]$}};
  \node (G) at (3.2,0){\scalebox{0.8}{$[\vb_1 \vb_2] \times [\wb_2 \wb_2]$}};
  \node (H) at (4.8,0){\scalebox{0.8}{$[\vb_0 \vb_2] \times [\vb_1 \wb_2]$}};
  \node (I) at (6.4,0){\scalebox{0.8}{$[\vb_1 \vb_0 \vb_2] \times [\wb_2]$}};
  \draw[-,thick] (A)--(B);
  \draw[-,thick] (A)--(C);
  \draw[-,thick] (A)--(D);
  \draw[-,thick] (B)--(E);
  \draw[-,thick] (B)--(F);
  \draw[-,thick] (B)--(G);
  \draw[-,thick] (C)--(H);
  \draw[-,thick] (C)--(I);
      \end{tikzpicture}
\end{center}
\caption{Pyramidal decomposition of a triangle and a rectangle with
  a common edge. The apex of the next base is always listed first.} 
\label{fig:TriangQuad}
\end{figure}

We now turn to the integration over the pyramids generated by
Algorithm~\ref{algo:pyradecomp} when applied to Cartesian products.
If $(\xb,\yb)$ is a point in an
iterated pyramid generated by the algorithm then 
\begin{equation*}
\xb = (1-\lambda) \ab + \lambda \xb_F\quad\text{and}\quad
\yb = (1-\lambda) \ab + \lambda \yb_F,
\end{equation*}
where $\lambda\in [0,1]$,
$\ab\in [\vb_0,\dots,\vb_s]$, $\xb_F \in F_x$ and $\yb_F \in F_y$ and 
$F_x \cap F_y = \emptyset$.
This makes clear that
\begin{equation*}
  \abs{\xb-\yb} = \lambda \abs{\xb_F-\yb_F},
\end{equation*}
where $\abs{\xb_F-\yb_F}$ is bounded away from zero and a
smooth function on the faces. These considerations lead to the main
result of this section.

\begin{theorem}\label{theo:intgr}
If $\mathcal{L}(P)$ is the set of leafs in the pyramidal
lattice of $P = P_x\times P_y$ then the following formula holds
\begin{multline}\label{decomp:kernelintgr}
\int\limits_{P_x}\int\limits_{P_y} \abs{\xb-\yb}^{-\alpha} g(\xb,\yb)
\, d\xb\yb =\\
\sum_{F_x\times F_y \in \leaves{P} \atop A_F \in \paths{F_x\times F_y}}\;
\!\!\!\!\!\!\delta_F
\int\limits_0^1 \int\limits_{A_F} \int\limits_{F_x} \int\limits_{F_y}
\abs{\xb_F-\yb_F}^{-\alpha} g(\xb,\yb) \, d\yb_F d\xb_F d\ab\,
(1-\lambda)^r \lambda^{s-\alpha} \, d\lambda.
\end{multline}
where $A$ is the convex hull of all apices on the path to $F_x\times
F_y$, $r = \dim(F_x) + \dim(F_y)$ and $s = \dim(P_x) + \dim(P_y) - r -
1$ and $\delta_F$ is defined as in \eqref{def:deltaAB} with
$B=F_x\times F_y$.
\end{theorem}
One can see that the integral is singular only when
$\lambda=0$, and that Gauss-Jacobi quadrature can be used to
effectively treat the singularity.

\section{Examples}\label{sec:examples}
This section describes two cases where one can determine a-priori 
the decompositions generated by Algorithm~\ref{algo:pyradecomp}: The Cartesian
product of two simplices and two cubes in any dimension.

\subsection{Cartesian Product of Two Simplices}\label{sec:two:simplices}
If $\{\vb_0,\dots,\vb_n\}$ are the vertices of the simplex $S$, then its faces
are the simplices spanned subsets of these vertices. 
More specifically, if $\I \subset \{ 0,\dots,
n\}$ is a subset of cardinality $r$ and $\I' =  \{ 0,\dots,n\}
\setminus \I$ is its complement, then we denote by
$S_\I$ the $n-r$ dimensional face of $S$ that is obtained by removing
the vertices indexed by $\I$ and spanned by the vertices indexed in $\I'$, i.e.,
\begin{equation*}
  S_\I = [\vb_{i_0},\dots \vb_{i_{n-r}}],\quad \I' =\{i_0,\dots, i_{n-r}\}.
\end{equation*}
Similar to \eqref{def:prdPoly} we write the second simplex in the
Cartesian product as
\begin{equation*}
  T = [\vb_0,\dots,\vb_k, \wb_{k+1},\dots \wb_m]
\end{equation*}
and denote by $T_\I$ the face of $T$ obtained by removing the vertices
indexed by $\I$.

Selecting $\vb_0\times \vb_0$ as the first apex in
Algorithm~\ref{algo:pyradecomp} applied to $S\times T$, results in 
two bases in the first level of the pyramidal lattice
\begin{equation*}
\singular{1} := \{ S_{\{0\}} \times T,\;  S \times T_{\{0\}}\}.
\end{equation*}
If $k=0$ then the two bases are nonsingular, and the decomposition is stopped.
Otherwise, 
both bases contain the singular vertex $\vb_1 \times \vb_1$. Using
that as the next apex leads to four faces in the second level of the
pyramidal lattice,
\begin{equation*}
  \singular{2} =
  \{ S_{\{0,1\}} \times T,\;  S_{\{0\}} \times T_{\{1\}},\;
     S_{\{1\}} \times T_{\{0\}}\;, S \times T_{\{0,1\}}\}.
\end{equation*}
By induction, one finds that the faces in the $\ell+1$-st level are
\begin{equation*}
  \singular{\ell+1} = \Big\{ S_\I \times T_{\I'} \Big\}_{\I \in \mathcal{P}(\{0,\dots,\ell \})}
\end{equation*}
where $\I' = \{0,\dots,\ell\} \setminus \I$ and
$\mathcal{P}(\{0,\dots,\ell \})$ is the power set of $\{0,\dots,\ell
\}$. If $\ell < k$ then for each face in $\singular{\ell+1}$ the next apex is
the singular vertex $\vb_{\ell+1}\times \vb_{\ell+1}$. If $\ell = k$ then all faces in
$\singular{k+1}$ are nonsingular. This proves the following result.

\begin{theorem}\label{theo:TwoSimplex}
The Cartesian product of the simplices
$S=[\vb_0,\dots,\vb_n]$ and  $T=[\wb_0,\dots,\wb_m]$ with common
vertices $\vb_i=\wb_i$, $i\in \{0,\dots,k\}$ has the following
decomposition into pyramids with nonsingular bases
\begin{equation*}
S \times T = \bigcup_{\I \in \mathcal{P}(\{0,\dots,k\})}
\pyr(\vb_0\times\vb_0,\dots.\vb_k\times\vb_k, S_\I \times T_{\I'} ).
\end{equation*}
Here, the number of nonempty iterated pyramids is $2^{k+1} - \sigma$, where
\begin{equation*}
  \sigma = \begin{cases}
             2 & \mbox{if } n=m=k,\\
             1 & \mbox{if } n>k, m=k \;\mbox{or}\; n=k, m>k, \\
             0 & \mbox{else.}
           \end{cases}
\end{equation*}
\end{theorem}

\subsection{Cartesian Product of Two Cubes}\label{sec:twoCubes}
In this section we consider the cubes
\begin{equation*}
  C_x = [0,1]^d  \quad \text{and} \quad C_y = [0,1]^k \times [-1,0]^{d-k}
\end{equation*}
which share the $k$-dimensional face $[0,1]^k \times
\Ob$. The discussion of this case turns out to be easier if the $x$-
and $y$-coordinates are rearranged in the order $(x_1,y_1,\dots,
x_d,y_d)$. Define the squares in the $xy$-plane
\begin{equation*}
  S = [0,1]^2  \quad \text{and} \quad  T = [0,1]\times[-1,0],
\end{equation*}
Then the product $C_x \times \C_y$ is equivalent to
\begin{equation}\label{def:twoCube}
P = S\times \dots \times S \times T \times \dots \times T,  
\end{equation}
where $S$ appears $k$ times and $T$ appears $d-k$ times.

The following vertices and faces of $S$ and $T$ will be important
later on
\begin{equation*}
  \begin{gathered}
    \vb_0 = (0,0),\; \vb_1 = (1,1),\; \vb_2 = (1,0),
    \; \vb_3 = (0,1),\\
    E_2 = \{1\} \times [0,1], \; E_3 = [0,1] \times \{1\},\;
    G_2 = \{1\} \times [-1,0],\;  G_3 = [0,1] \times \{-1\},
  \end{gathered} 
\end{equation*}
see Figure~\ref{fig:squares}. Note that the vertices  $\vb_0$ and
$\vb_1$ and the edges $E_2$, $E_3$ are singular, because they
intersect or are on
the line $x=y$.

The polytope $P$ has $2^k$ singular vertices which are given by
$\vb_\sigmab = \vb_{\sigma_1} \times \dots \times \vb_{\sigma_d}$ where
$\sigmab = (\sigma_1,\dots,\sigma_d)$ and
\begin{equation*}
  \sigma_i \in \begin{cases} \{0,1\} & \text{if}\; i \leq k,\\
                 \{0\} &  \text{if}\; i > k.
               \end{cases}
\end{equation*}

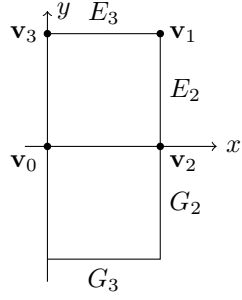
\begin{figure}
\begin{center}
\begin{tikzpicture}[xscale=1.5, yscale=1.5]
  \draw (0,-1) -- (1,-1) -- (1,1) -- (0,1) ;
  \draw[->] (-0.2,0)--(1.5,0) node[right]{$x$};
  \draw[->] (0,-1.2)--(0,1.2) node[right]{$y$};
\coordinate[label = below left:$\vb_0$] (A) at (0,0);
\coordinate[label = below right:$\vb_2$] (B) at (1,0);
\coordinate[label = right:$\vb_1$] (C) at (1,1);
\coordinate[label = left:$\vb_3$] (D) at (0,1);
\coordinate[label = right:$E_2$] (E) at (1,0.5);
\coordinate[label = above:$E_3$] (F) at (0.5,1);

\coordinate[label = right:$G_2$] (G) at (1,-0.5);
\coordinate[label = below:$G_3$] (H) at (0.5,-1);
\node at (A)[circle,fill,inner sep=1pt]{};
\node at (B)[circle,fill,inner sep=1pt]{};
\node at (C)[circle,fill,inner sep=1pt]{};
\node at (D)[circle,fill,inner sep=1pt]{};
\node at (E){};
\node at (F){};
\end{tikzpicture}
\end{center}
\caption{Vertices and faces of $S$ and $T$.} 
\label{fig:squares}
\end{figure}

The faces of $P$ are $d$-fold tensor products of the faces
of $S$ and $T$. We will see soon that the pyramidal decomposition only
involves faces of the form $F= F_1 \times \dots \times
F_d$, where
\begin{equation}\label{twoCubes:faces}
F_i \in 
  \begin{cases}
\{S, E_2, E_3, \vb_2, \vb_3 \} &\mbox{ if } i\leq k,\\
\{T, G_2, G_3\} &\mbox{ if } i>k.
\end{cases}
\end{equation}
To characterize products of faces in \eqref{twoCubes:faces}, we count
edges and vertices as follows
\begin{equation*}
  \begin{aligned}
    n_E(F) &= \# \big\{ i \,:\, F_i \in  \{E_2, E_3\} \big\},\\
    n_G(F) &= \#  \big\{ i \,:\, F_i \in  \{G_2, G_3\}\big\},\\
    n_V(F) &= \#  \big\{ i \,:\, F_i \in  \{\vb_2,\vb_3\}\big\}\\
  \end{aligned}
\end{equation*}
A product of faces in \eqref{twoCubes:faces} is singular if all
$F_i$'s are in $\{ S, E_2,
E_3, T\}$. We group these by the number of edges 
\begin{equation*}
  \singular{\ell} = \Big\{ F_1\times\dots\times F_d \,:\, F_i \in
    \{S,E_2,E_3\},\; i\leq k,\; F_i = T,\; i>k, \\
    \mbox{ and } n_E(F) = \ell \Big\}.
\end{equation*}
Note that $\singular{0} = \{ P \}$.
For $F\in \singular{\ell}$ the apex $\vb_\sigmab$ is
chosen according to the rule
\begin{equation}\label{twoCubes:apex}
  \sigma_i = \begin{cases} 1 & \mbox{if } F_i \in \{ E_2, E_3 \},\\
                           0  & \mbox{if } F_i \in \{ S, T \}.
  \end{cases}
\end{equation}
With this choice of apex the set $\faces{F, \vb_\sigmab}$ consists of
faces $F'$ of the following form
\begin{equation}\label{twoCubes:bases}
  \exists i:\;
  F'_{i} = \begin{cases}
             E_2 \mbox{ or } E_3 & \mbox{if } F_i = S,\\
             \vb_2  & \mbox{if } F_i = E_2,\\
             \vb_3  & \mbox{if } F_i = E_3,\\
              G_2 \mbox{ or } G_3 & \mbox{if } F_i = T,\\
           \end{cases}
  \quad\mbox{and}\quad F'_{i'} = F_{i'}\quad  \mbox{for}\; i' \not= i.   
\end{equation}
If in the above formula $F_i = S$ then $F'$ is a singular face in
$\singular{\ell+1}$. If $F_i=E_2$ or $F_i=E_3$ then $F'$
is nonsingular and in the set
\begin{equation*}
\begin{aligned}
\bases{\ell+1}_1 = \Big\{ F_1\times&\dots\times F_d \,:\, 
    F_i \in \{S, E_2, E_3, \vb_2,\vb_3\}, \; i\leq k,\; F_{i'} = T,\; i' > k,\\
    &  \mbox{ and } n_V(F) = 1,\; n_E(F) = \ell - 1 \Big\}.
  \end{aligned}
\end{equation*}
If in \eqref{twoCubes:bases} $F_i = T$ then $F'$ is
nonsingular and in the set
\begin{equation*}
\begin{aligned}
\bases{\ell+1}_2 = \Big\{ F_1\times&\dots\times F_d \,:\, 
    F_i \in \{S, E_2, E_3\},\; i\leq k,\;F_i \in \{T, G_2,G_3\},\; i' > k, \\
   & \mbox{ and }n_G(F) = 1,\;  n_E(F) = \ell \Big\}.
  \end{aligned}
\end{equation*}
By induction, one sees that the singular faces in the pyramidal
lattice are the union of sets $\singular{\ell}$.
Moreover, the bases are the union of the sets $\bases{\ell}_1$ and   
$\bases{\ell}_2$, $1\leq \ell \leq k$.
The cardinality of all bases in the pyramidal lattice can now be
determined with simple counting. We find that
\begin{equation*}
\begin{aligned}
\# \mathcal{B}_1 &= k \sum_{\ell=1}^{k} \binom{k-1}{\ell-1} 2^{\ell} =
    2 k 3^{k-1},\\
\# \mathcal{B}_2 &=  2(d-k) \sum_{\ell=0}^{k} \binom{k}{\ell-1} 2^{\ell} =
    2 (d-k) 3^{k}.
\end{aligned}
\end{equation*}
Consider now $F\in \bases{\ell}_1 \cup \bases{\ell}_2$. Denote by
$\sigmab = (\sigma_1,\dots,\sigma_d)$ the positions of the edges and
vertices, that is,
\begin{equation*}
  \sigma_i =
  \begin{cases}
      0 & \mbox{ if } F_i \in \{ S, T \},\\             
      1 & \mbox{ if } F_i \in \{ E_2,E_3, G_2, G_3, \vb_2, \vb_3 \}.            
  \end{cases}
\end{equation*}
Then all paths of apices in the pyramidal lattice that lead from $P$ to $F$ are
of the form $\vb_\Ob, \vb_{\sigmab_1},\dots, \vb_{\sigmab_{\ell}}$,
where $\sigmab_j = ( \sigma_{j1},\dots \sigma_{jd})$, $\sigma_{ji} \in
\{0,1\}$ and
\begin{equation*}
\sum_{i=1}^d \sigma_{ji} = j,\quad\mbox{and}\quad \sigma_{j+1 i} \geq \sigma_{ji}.
\end{equation*}
As described in Section~\ref{sec:triangulations} these vertices are a
triangulation of 
\begin{equation}\label{def:AF}
A_F = I_{\sigma_1} \times \dots \times I_{\sigma_d},
\end{equation}
where $I_0 = \{(\vb_0,\vb_0)\}$ and $I_1 = [(\vb_0\vb_0),(\vb_1\vb_1)]$.
This is a $\ell$-dimensional cube and a subset of \eqref{def:twoCube}.

The construction in this section proves the following result.
\begin{theorem}\label{theo:TwoCubes}
The Cartesian product of the cubes
$C_x = [0,1]^d$ and  $C_y = [0,1]^k \times [-1,0]^{d-k}$ has the following
decomposition of convex hulls of singular and nonsingular faces.
\begin{equation*}
  C_x \times C_y =
  \bigcup_{\ell=2}^{k+1} \bigcup_{F \in \bases{\ell}_1} \conv(A_F, F)
  \quad \cup \;\quad
  \bigcup_{\ell=1}^{k+1} \bigcup_{F \in \bases{\ell}_2} \conv(A_F, F)  .
\end{equation*}
The total number of terms is $(6d-4k)3^{k-1}$.
\end{theorem}

\subsection{Comparison with the Prior Method}\label{sec:comp}
Singular integrals of the type \eqref{def:kernelintgr} for the case of
two simplices and two cubes were previously discussed in a number of
papers by Chernov, Schwab and co-workers ~\cite{chern-peterd-schwab11,
  chernov-schwab12, chernov-reinarz13, chern-peterd-schwab15}.  Their
method is different from the present method in that it is based on the
transformation $\zb =\xb-\yb$, which 
maps the singularity to $\zb=\mathbf{0}$. This will work as long as the
projection of the polytopes into the $\zb$-plane has a manageable
geometry, which is the case for simplices and cubes. This projection
is then decomposed into simpler domains that are parameterized in a
way that the singularity can be incorporated in the weight function of
the Gauss-Jacobi quadrature.  The final result is a formula similar to
\eqref{decomp:kernelintgr}, although in that decomposition it is
difficult to trace which part of $P_x\times P_y$ each integral belongs
to as the transformations can get quite complicated.

It is interesting to compare the number of terms of the previous work
with the present method. In \cite[formula (3.37)]{chern-peterd-schwab11}, the number
for two simplices is
\begin{equation*}
  K = \begin{cases}
        2 & \text{ for } k=0\\
        3(2^{k+1}-2) & \text{ for } 1\leq k \leq d-1\\
        2^{k+1} - 2 & \text{ for } k=d
        \end{cases}
\end{equation*}
and the number for two cubes is      
\begin{equation*}
K = 2^k(2d-k).
\end{equation*}
These counts have to be compared with Theorems~\ref{theo:TwoSimplex}
and ~\ref{theo:TwoCubes}. For convenience, we compare the values of
$K$ for low dimensions in Tables~\ref{tab:esaimSplx} and~\ref{tab:esaimCube}. 
\begin{table}
\begin{center}
\begin{tabular}{l|ccccc|ccccc}  
\hline                                                                
d & k=0 & k=1 & k=2 & k=3 & k=4 & k=0 & k=1 & k=2 & k=3 & k=4 \\ 
  \hline
  1 & 2 & 2 &     &    &    & 2 & 2 \\
  2 & 2 & 6 & 6   &    &    & 2 & 4 & 6 \\
  3 & 2 & 6 & 18  & 14 &    & 2 & 4 & 8 & 14\\
  4 & 2 & 6 & 18  & 42 & 30 & 2 & 4 & 8 & 16 & 30\\
\hline                                                                
\end{tabular}
\caption{Number of integrals for two simplices. Left:
  method of~\cite{chern-peterd-schwab11}. Right: present method}
\label{tab:esaimSplx}
\end{center}
\end{table}

\begin{table}
\begin{center}
\begin{tabular}{l|ccccc|ccccc}  
\hline                                                                
d & k=0 & k=1 & k=2 & k=3 & k=4 & k=0 & k=1 & k=2 & k=3 & k=4 \\
  \hline
  1 & 2 & 2  &    &    &    & 2 &  2 \\
  2 & 4 & 6  & 8  &    &    & 4 &  8 & 12\\
  3 & 6 & 10 & 16 & 24 &    & 6 & 14 & 30 & 54 \\
  4 & 8 & 14 & 24 & 40 & 64 & 8 & 20 & 48 & 108 & 216 \\
\hline                                                                
\end{tabular}
\caption{Number of integrals for two cubes. Left:
  method of~\cite{chern-peterd-schwab11}. Right: present method}
\label{tab:esaimCube}
\end{center}
\end{table}
For the simplex, one can see that except in the self case and the case
$k=1$ the present method has a lower count. On the other hand, for the cube the
counts in~\cite{chern-peterd-schwab11} are lower. However,
the original method is limited to parallelotopes,
whereas the present method is more general as it applies to any polytopes that are
combinatorially equivalent to cubes. 

\section{Numerical Examples}\label{sec:numresult}
The method has been implemented for the case of two simplices of
arbitrary dimension with any number of common vertices. As discussed
in Section~\ref{sec:two:simplices},
the faces $A_F$, $F_x$ and $F_y$ in the integrals of
Theorem~\ref{theo:intgr} are all simplices of various dimensions. If
$T = [\vb_0,\dots,\vb_n]$ is such a simplex then
\begin{equation*}
  \xb(\hat \xb) = \vb_0 + \sum_{k=1}^n (\vb_k - \vb_{k-1} )\hat x_k 
\end{equation*}
is a parameterization $\xb : \hat T_n\to T$ where $\hat T_n$ is the
$n$-dimensional standard simplex
\begin{equation*}
  \hat T_n = \{ \hat \xb \in \R^n : 0 \leq \hat x_n \leq \dots \leq \hat x_1 \leq 1 \}.
\end{equation*}
We consider two methods to construct quadrature rules for $\hat
T_n$. The first is based on the well known Duffy transform, which maps
the unit cube $[0,1]^n$ to $\hat T_n$ via
\begin{equation*}
  \hat \xb(\xib) =
  (\xi_1,\, \xi_1 \xi_2, \dots, \xi_1 \xi_2\cdots\xi_n).
\end{equation*}
The Jacobian of this transformation is $J(\xib) = \xi_1^{n-1} \xi_2^{n-2} \cdots \xi_{n-1}$. 
The transformed integral is approximated by
a tensor product Gauss-Jacobi rule where the $J(\xib)$ is included in
with the weight function. This leads to
\begin{equation}\label{qr:TP}
  \int\limits_{\hat T_n} \varphi(\hat \xb)\, d\hat\xb \approx
  \sum_{q_1=1}^{p}\cdots \sum_{q_n=1}^{p}
  \varphi\left(\xi_{1 q_1},\, \xi_{1 q_1} \xi_{2 q_2}, \dots,\xi_{1 q_1}
    \cdots \xi_{n q_n}\right) 
  w_{1 q_1} \cdots w_{n q_n}
\end{equation}
where $\xi_{j q}$ and $w_{j q}$ are the nodes and weights of the
Gauss-Jacobi rule for the weight function $w_j(\xi_j) = \xi_j^{n-j}$.
The tensor product form of \eqref{qr:TP} implies that the quadrature
rule is exact for monomials of the form $\xb^\alpha$, with
$\alpha_i \leq 2p-1$, which is more than what is necessary to integrate
multivariate polynomials of degree $2p-1$ exactly.

Thus the tensor product rule suboptimal in the number of nodes, and
accumulates quadrature points near the origin.  To obtain rules with
fewer nodes and better node distributions one has to solve a
polynomial system for the nodes $\xb_k\in \hat T_n$ and weights
$w_q> 0$ such that the rule
\begin{equation}\label{qr:GG}
  \int\limits_{\hat T_n} \varphi(\hat \xb)\, d\hat\xb \approx
  \sum_{q=1}^{N_q}
  \varphi\left(\xb_q \right) 
  w_q
\end{equation}
is exact for all polynomials up to a given degree. This is known as
generalized Gauss quadrature. Except for some special cases it is
unknown what the least number of nodes is \cite{cools97}. The tensor
product rule is exact, but has more nodes than the number of
conditions in \eqref{qr:GG} suggests. More efficient rules can be
obtained by a elimination procedure~\cite{xiao-gimbutas10}. This
process iteratively eliminates nodes and applies a nonlinear solver to
satisfy the exactness condition for the reduced rule. In our numerical
experiments we have used the rules that were generated by
implementation in~\cite{slobodkins-tausch21} which is described
in~\cite{slobodkins-tausch23}.

Figure~\ref{fig:self} displays the error versus the degree of
the quadrature rule for the self interaction of a three-dimensional
and a five-dimensional simplex. Thus the total dimension of the
integral \eqref{def:kernelintgr} is six and ten, respectively. In
Assumption~\ref{asu:kernel}, we set $g(\xb,\yb)=\exp(\sum_i x_i+y_i)$
and $\alpha = 1$. One can see that in all cases the error decays at
the same exponential rate. The tensor product for the odd order $2p+1$
is identical to order $2p$, which explains the staircase behavior in
the plot.  The generalized Gauss rules for even orders appear to give
somewhat better errors than the tensor product rule whereas for odd
degrees the difference between the two quadrature schemes are
negligible.  The highest degree computed is 22 in three dimensions and
14 in five dimensions. This limitation comes from the CPU time and
memory requirements to generate the generalized Gauss rules.
The reference value
for the error computation is the value of the integral obtained with
the highest degree of precision.

The second plot in Figure~\ref{fig:self} displays the same errors
versus the number of function evaluations. Here one can see the
increased cost associated with the dimension. The generalized Gauss
rules reduce the cost significantly while maintaining the accuracy of
the result.

Figure~\ref{fig:otherSng} displays the error versus the number of
function evaluations for different singular cases. The figure
illustrates that number of function value evaluations is roughly
proportional to the number of terms in the
decomposition~\ref{theo:TwoSimplex}. Both plots show data obtained
with the generalized Gauss rules.

\begin{figure}
  \begin{center}
    \includegraphics[width=2.5in]{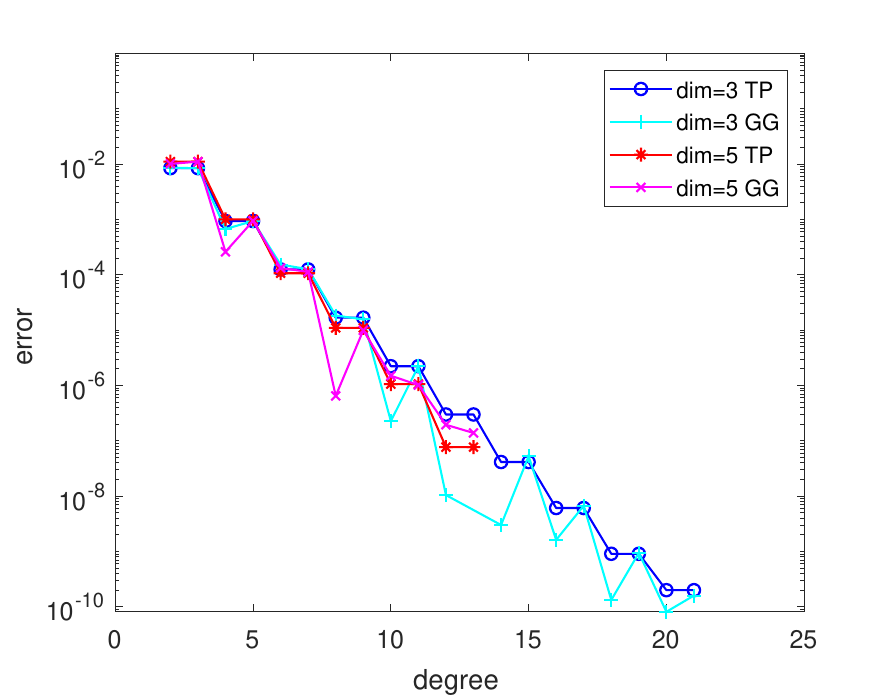}
    \includegraphics[width=2.5in]{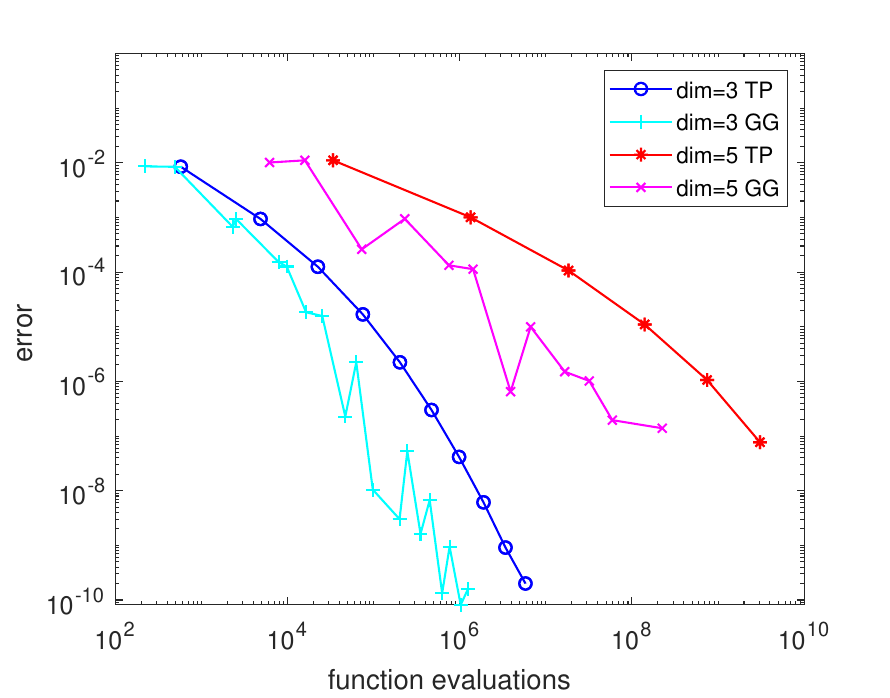}
\end{center}
\caption{Quadrature errors vs. the order (left) and the number of
  function evaluations. Self case of the three and five
  dimensional simplex. } 
\label{fig:self}
\end{figure}

\begin{figure}
\begin{center}
    \includegraphics[width=2.5in]{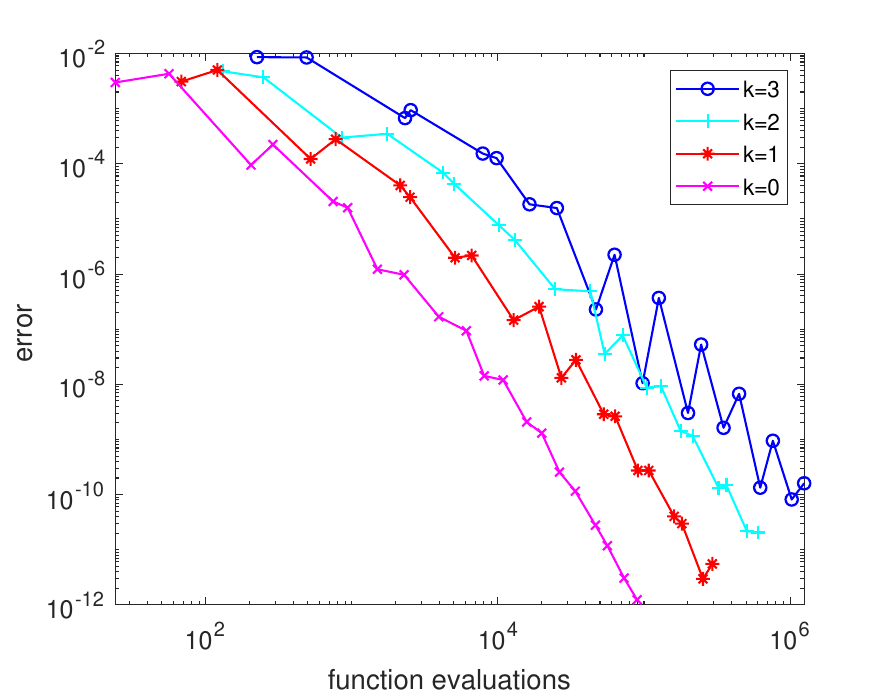}
    \includegraphics[width=2.5in]{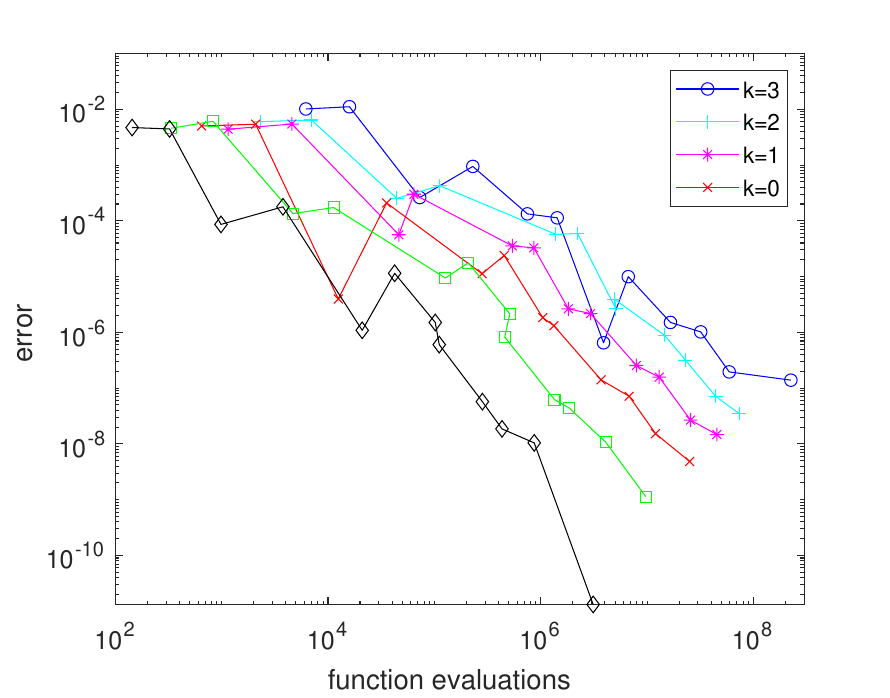}
\end{center}
\caption{Quadrature errors vs. number of function
  evaluations. Different singular cases of the three dimensional
  simplex (left) and the five dimensional simplex (right).} 
\label{fig:otherSng}
\end{figure}

\section{Conclusions} The decomposition algorithm combined with the
Gauss-Jacobi quadrature for the resulting convex hulls is a simple
method to construct effective quadrature rules for general polytopes
with singular functions. Previous constructions are limited to special
cases. The exponential convergence of the quadrature for analytic
kernels can be explained by approximation properties of multivariate
polynomials. However, a more thorough analysis is still needed, for
instance, to optimize the quadrature of the iterated integrals of the
convex hulls and to handle adaptivity for polytopes with poor shape
quality.


\begin{thebibliography}{10}

\bibitem{bonito-lei-pasciak19}
A.~Bonito, W.~Lei, and J.~E. Pasciak.
\newblock Numerical approximation of the integral fractional {L}aplacian.
\newblock {\em Numerische Mathematik}, 142:235--278, 2019.

\bibitem{chernov-reinarz13}
A.\ Chernov and A.\ Reinarz.
\newblock Numerical quadrature for high-dimensional singular integrals over
  parallelotopes.
\newblock {\em Computers and Mathematics with Applications}, 66(7):1213--1231,
  2013.

\bibitem{chernov-schwab12}
A.\ Chernov and C.\ Schwab.
\newblock Exponential convergence of {G}auss-{J}acobi quadratures for singular
  integrals over simplices in arbitrary dimension.
\newblock {\em SIAM J. Numer. Anal.}, 50(3):1433–--1455, 2012.

\bibitem{chern-peterd-schwab11}
A.~Chernov, T.~von Petersdorff, and C.~Schwab.
\newblock Exponential convergence of hp quadrature for integral operators with
  {G}evrey kernels.
\newblock {\em ESAIM: M2AN}, 45:387--422, 2011.

\bibitem{chern-peterd-schwab15}
A.~Chernov, T.~von Petersdorff, and C.~Schwab.
\newblock Quadrature algorithms for high dimensional singular integrands on
  simplices.
\newblock {\em Numer. Algor.}, 70:847--874, 2015.

\bibitem{cools97}
R.~Cools.
\newblock Constructing cubature formulae: the science behind the art.
\newblock {\em Acta Numerica}, pages 1--54, 1997.

\bibitem{delia-etal20}
M.~D’Elia, Q.~Du, C.~Glusa, M.~Gunzburger, X.~Tian, and Z.~Zhou.
\newblock Numerical methods for nonlocal and fractional models.
\newblock {\em Acta Numerica}, pages 1--124, 2020.

\bibitem{erichsen-sauter98}
S.~Erichsen and S.A.\ Sauter.
\newblock Efficient automatic quadrature in 3-{D} {G}alerkin {BEM}.
\newblock {\em Computer Methods in Appl.\ Mech and Engrg.}, 157:215---224,
  1998.

\bibitem{feist:bebebdorf23}
B.~Feist and M.~Bebendorf.
\newblock Fractional {L}aplacian -- quadrature rules for singular double
  integrals in 3{D}.
\newblock {\em Computational Methods in Applied Mathematics}, 23:623--645,
  2023.

\bibitem{goodman-orourke04}
J.E. Goodman and J.~O'Rourke, editors.
\newblock {\em Handbook of Discrete and Computational Geometry}.
\newblock Chapman and Hall/CRC, 2nd edition, 2004.

\bibitem{griebel:hamaekers07}
M.~Griebel and J.~Hamaekers.
\newblock Sparse grids for the {S}chr\"odinger equation.
\newblock {\em ESAIM: Math. Modelling Numer. Anal.}, 41(2):215–247, 2007.

\bibitem{grunbaum03}
B.\ Gr\"unbaum.
\newblock {\em Convex Polytopes}, volume 221 of {\em Graduate Texts in
  Mathematics}.
\newblock Springer, 2003.

\bibitem{harbr:wendl:zorii16}
H.~Harbrecht, W.L. Wendland, and N.~Zorii.
\newblock Rapid solution of minimal riesz energy problems.
\newblock {\em Numer. Meth. Partial Diff. Eq.}, 32:1535–1552, 2016.

\bibitem{hsiao-wendland08}
G.~C. Hsiao and W.~L. Wendland.
\newblock {\em Boundary Integral Equations}, volume 168 of {\em Applied
  Mathematical Sciences}.
\newblock Springer, 2008.

\bibitem{kaibel-pfetsch02}
V.~Kaibel and M.~E. Pfetsch.
\newblock Some algorithmic problems in polytope theory.
\newblock {\em arXiv preprint}, 2002.
\newblock math/0202204.

\bibitem{mason-tausch19}
N.~Mason and J.~Tausch.
\newblock Quadrature for parabolic {G}alerkin {BEM} with moving surfaces.
\newblock {\em Comput. Math. Appl.}, 77(1):1--14, 2019.

\bibitem{mohyaddin-tausch23}
S.~Mohyaddin and J.~Tausch.
\newblock A fast method for evaluating volume potentials in the {G}alerkin
  boundary element method.
\newblock {\em SIAJS}, 45(2):A480--A501, 2023.

\bibitem{of:wend:zorii10}
G.~Of, W.L. Wendland, and N.~Zorii.
\newblock On the numerical solution of minimal energy problems.
\newblock {\em Complex Var. Elliptic Equ.}, 55(11), 2010.

\bibitem{poltz-schanz19}
D.~P\"olz and M~Schanz.
\newblock Space-time discretized retarded potential boundary integral
  operators: Quadrature for collocation methods.
\newblock {\em SIAM J. Sci. Comput.}, 41(6):A3860–A3886, 2019.

\bibitem{sauter-schwab11}
S.\ Sauter and C.\ Schwab.
\newblock {\em Boundary Element Methods}.
\newblock Springer, 2011.

\bibitem{seibel23}
D.~Seibel.
\newblock Almost complete analytical integration in {G}alerkin boundary element
  methods.
\newblock {\em SIAM J. Sci. Comput.}, 45(4):A2075--A2100, 2023.

\bibitem{slobodkins-tausch21}
A.~Slobodkins.
\newblock gen-quad.
\newblock https://github.com/arkslobodkins/gen-quad.

\bibitem{slobodkins-tausch23}
A.~Slobodkins and J.~Tausch.
\newblock A node elimination algorithm for cubature of high-dimensional
  polytopes.
\newblock {\em Computers {\&} Math. Appl.}, 144:229--236, 2023.

\bibitem{tausch22}
J.~Tausch.
\newblock Adaptive quadrature rules for {G}alerkin {BEM}.
\newblock {\em Computers Math. Appl.}, 113:270--281, 2022.

\bibitem{xiao-gimbutas10}
H.\ Xiao and Z.\ Gimbutas.
\newblock A numerical algorithm for the construction of efficient quadrature
  rules in two and higher dimensions.
\newblock {\em Comput. Math. Appl.}, 59(2):663--676, 2010.

\bibitem{ziegler95}
G.~Ziegler.
\newblock {\em Lectures on Polytopes}, volume 152 of {\em Graduate Texts in
  Mathematics}.
\newblock Springer, 1995.

\end{thebibliography}

\end{document}